\documentclass[10pt]{amsart}
\usepackage[font=small,textfont=sc,labelfont=sc]{caption}
\usepackage{tikz,shuffle,graphicx,multicol, enumerate,amssymb, mathtools,amsmath}
\usepackage{subfig, multirow, pifont}
\usepackage{color}
\usepackage{setspace}
\usepackage{fullpage}

\newtheorem{thm}{Theorem}

\newtheorem{lemma}[thm]{Lemma}
\newtheorem{conjecture}{Conjecture}

\newtheorem{remark}[thm]{Remark}

\theoremstyle{remark}
\newtheorem{example}[thm]{Example}

\numberwithin{equation}{section}

\makeatletter
\def\testb#1{\testb@i#1,,\@nil}%
\def\testb@i#1,#2,#3\@nil{%
  \draw[-, thick] (O) --++(#1);
  \ifx\relax#2\relax\else\testb@i#2,#3\@nil\fi}
\makeatother   

\newcommand{\makediag}[1]{
    \coordinate (O) at (0,0); \coordinate (N) at (0,1);
    \coordinate (NE) at (0.98,0.98); \coordinate (E) at (1,0);
    \coordinate (SE) at (0.98,-0.98); \coordinate (S) at (0,-1);
    \coordinate (SW) at (-0.98,-0.98);\coordinate (W) at (-1,0);
    \coordinate (NW) at (-0.98,0.98); \coordinate (B1) at (1.2,1.2);
    \coordinate (B2) at (-1.2,-1.2);
    \testb{#1}
} 
\newcommand{\diagr}[1]{
  \begin{tikzpicture}[scale=0.15]\makediag{#1}\end{tikzpicture}
}

\newcommand{\mA}{\ensuremath{\mathcal{A}}}
\newcommand{\mS}{\ensuremath{\mathcal{S}}}
\newcommand{\mU}{\ensuremath{\mathcal{U}}}

\newcommand{\mB}{\ensuremath{\mathcal{B}}}

\newcommand{\mQ}{\ensuremath{\textsc{Q}}}
\newcommand{\mR}{\ensuremath{\textsc{R}}}
\newcommand{\mH}{\ensuremath{\textsc{H}}}

\newcommand{\veca}{\mathbf{a}}
\newcommand{\vecx}{\mathbf{x}}

\begin{document}

\title[Lattice Path Asymptotics]{A Combinatorial Understanding of Lattice Path Asymptotics}
\author{Samuel Johnson}
\address{Department of Mathematics\\ Simon Fraser University\\ 8888 University Drive\\ Burnaby, British Columbia\\ V5A1S6, Canada}
\email{samuelj@sfu.ca}

\author{Marni Mishna}
\address{Department of Mathematics\\ Simon Fraser University\\ 8888 University Drive\\ Burnaby, British Columbia\\ V5A1S6, Canada}
\email{mmishna@sfu.ca}
\thanks{Partial support by NSERC Discovery grant program, and  
mprime project Mathematics of Computer Algebra and Analysis (MOCAA)}

\author{Karen Yeats}
\address{Department of Mathematics\\ Simon Fraser University\\ 8888 University Drive\\ Burnaby, British Columbia\\ V5A1S6, Canada}
\email{karen\_yeats@sfu.ca}


\date{}

\begin{abstract}
  We provide a combinatorial derivation of the exponential growth
  constant for counting sequences of lattice path models restricted to
  the quarter plane.  The values arise as bounds from analysis of
  related half planes models. We give explicit formulas, and the
  bounds are provably tight. The strategy is easily generalized to
  cones in higher dimensions, and has implications for random
  generation.
\end{abstract}

\maketitle

\section{Introduction}
Lattice path models have enjoyed a sustained popularity in mathematics
over the past century, owing in part to their simplicity and ease of
analysis, but also their wide applicability both in mathematics,
physics, and chemistry.  The basic enumerative question is to
determine the number of walks of a given length in a given model. The
past ten years have seen many interesting developments in the
asymptotic and exact enumeration of lattice models, with new
techniques coming from computer algebra, complex analysis and algebra.
A first approximation to this value is the \emph{exponential growth
  constant}, also called the \emph{connective constant}, which itself
carries combinatorial and probabilistic information. For example, it
is directly related to the limiting free energy in statistical
mechanical models.

A model is defined by the steps that are allowed, and the region to
which the walks are restricted (generally cones and
strips). Particular focus has been on small step models (where the
steps are a subset of $\{0,\pm 1\}^2$) restricted to~$\mathbb{Z}_{\geq
  0}^2$, and general approaches versus resolution of individual
cases. For example, three distinct strategies for asymptotic
enumeration have recently emerged. Fayolle and Raschel~\cite{FaRa12}
have determined expressions for the growth constant for small step
models using boundary value problem techniques. Recast as diagonals,
techniques of analytic combinatorics of several variables apply to
some of the models with D-finite generating functions ~\cite{MeMi16,
 MeWi16}. Finally, the important sub-class of excursions, that is, of
walks which return to the origin, are well explored via the
probability work of Denisov and Wachtel~\cite[Section
1.5]{DeWa15}. Bostan, Raschel and Salvy~\cite{BoRaSa14} made their
results explicit in the enumeration context. Most of these asymptotic
results are obtained with machinery which does not sustain a clear
underlying combinatorial picture.

Many of these results exclude \emph{singular\/} models: A two
dimensional model is singular if the support of the step set is
contained in a half plane. Many singular models are either trivial or
reduce to a problem in a lower dimension. The singular models are
considered in~\cite{MiRe09, MeMi14a}.

This paper provides a formula for an upper bound on the growth
constant of the counting sequence for lattice models restricted to a
convex cone, with intuitive combinatorial interpretations of intermediary
computations. Our formula is most explicit in the case of nonsingular
2-dimensional walks restricted to the first quadrant, but is valid
\emph{for all models}. Our general
strategy is based on the following basic observation:
\begin{quotation}\noindent\emph{In any lattice path model, the set of
    walks restricted to the first quadrant is a subset of the walks
    restricted to some half plane which contains that
    quadrant. Consequently, for any fixed length, the number of walks
    in that half plane is an upper bound for the number of walks in
    the quarter plane. }
\end{quotation}
Bounds on walks in half planes are readily computable, for example
using the results of Banderier and Flajolet~\cite{BaFl02}.
Remarkably, we are able to give \emph{tight\/} bounds on the growth
constant by considering all of the half planes that contain the
quarter plane. Furthermore, our bounds are insightfully tight in that
they give a single simple combinatorial interpretation of the multiple
cases treated by Fayolle and Raschel~\cite{FaRa12}.  Our one idea
\emph{unifies\/} their cases, which depend on various parameters of
the model. Furthermore, our approach also applies to singular models.
We use only the elementary calculus observation that a minimum of a
real valued continuous function~$f$ with domain~$D$ must occur either
at the boundary of~$D$ or at a critical point~$\tau \in D$ satisfying
$f'(\tau) = 0$. Our approach is combinatorial and readily adaptable to
models with larger steps, weighted steps, and to models in higher
dimensions. Furthermore, there are implications for random generation,
as we discuss in the conclusion.

In an earlier version of this article we conjectured that our bounds
were tight. This led to a proof by Garbit and Raschel~\cite{GaRa16}
that the bounds we find in the nonsingular case actually are
tight. Simultaneously, and independently, similar results were proved
by Duraj~\cite{Dura14}. Now, Garbit, Mustafa and
Raschel~\cite{GaMuRa16} conjecture that in some cases the
sub-exponential growth also matches that of the minimizing half-plane,
futher validating our interpretation.

\subsection{Conventions and notation}
Here, a \emph{lattice path model\/} is a combinatorial class denoted
by~$\mR(\mS)$ which is defined by a convex cone~$\mR$, and a finite
multiset of allowable steps (vectors), $\mS$. We focus on regions that
are half-planes through the origin, and the first quadrant
$\mathbb{Z}_{\geq 0}^2$.  We restrict~$\mS$ to be a finite subset of
$\mathbb{Z}^2$.  A walk of length $n$, say $w=w_0w_1 \dots w_n$, is a sequence of points
$w_i\in\mR$ such that $w_i-w_{i-1}\in \mS$ and for $i=1..n$.
We denote by $\mR(\mS)_n$ the subset of all walks of length~$n$ in
$\mR(\mS)$.

The central quantity we investigate is the number of walks with~$n$
steps in a given model, $|\mR(\mS)_n|$.  We write
$\mH=\mathbb{R}\times\mathbb{R}_{\geq 0}$ for the upper half plane
and~$\mQ=\mathbb{R}_{\geq 0}\times\mathbb{R}_{\geq 0}$ for the first
quadrant, and abbreviate $h_n = |\mH(\mS)_n|$ and $q_n = |\mQ(\mS)_n|$
when $\mS$ is clear.  In this work we focus on models in~$\mQ$.

A step set is said to be made of \emph{small steps\/} if $\mS\subseteq
\{0, \pm1\}^2\setminus\{(0,0)\}$ and in this case we use the compass
abbreviations $NW\equiv(-1,1), N\equiv(0,1), NE\equiv(1,1)$, etc.  We
might also consider larger regions, and more general step sets. We say
a model $\mR(\mS)$ is \emph{nontrivial\/} if it contains at least one
walk of positive length, and if for every boundary of~$\mR$, there
exists an unrestricted walk on $\mS$ which crosses that boundary at
some point other than an intersection of boundaries (i.e. in two
dimensions, not at the origin). The \emph{excursions\/} are the
sub-class consisting of walks which start and end at the origin. A
step set is said to be \emph{singular} if it is contained within a
single half-plane.

The (exponential) \emph{growth constant} of the sequence is defined as
the limit
\[
K_\mS=\lim_{n \rightarrow \infty} q_n^{1/n}.
\]
The limit exists by a classic argument: $q_n\cdot q_m\leq q_{n+m}$, so
by a Theorem of Hill~\cite{Hill48} the limit exists and hence (see
\cite[Theorem IV.7]{FlSe09}) it must be the reciprocal of the dominant
singularity of the generating function of the~$q_n$.
The present work determines bounds for the growth
constant, $K_\mS$, of the sequence~$\{q_n\}$.

Our strategy uses the simple relation that if $\mQ\subset \mR$, then
$\mQ(\mS)_n\subset \mR(\mS)_n$ and hence $q_n\leq|\mR(\mS)_n|$. This
is true for all~$n$, hence it is also true that
$\lim_n q_n^{1/n} \leq \lim_n |\mR(\mS)_n|^{1/n}$. It turns out, by
considering well chosen regions, we are able to perfectly bound the
growth constant~$K_\mS$.  Our preferred bounding regions are the half
planes
\[\mH_{\theta}=\{(x,y): x\sin\theta+y\cos\theta\geq 0 \}\] where~$\theta
\in [0,\pi/2]$.  We denote by $K_\mS(\theta)$ the growth constant of
the sequence of the number of walks of length $n$ in this region:
\[
K_\mS(\theta) = \lim_{n\rightarrow\infty} |\mH_\theta(\mS)_n|^{1/n}.
\]
Using the relation
\[
K_\mS \leq K_\mS(\theta) \quad \text{for all } 0\leq \theta \leq \pi/2,
\]
we can deduce bounds on $K_\mS$, by finding explicit expressions for
$K_\mS(\theta)$. 

\begin{figure}                 
\centering		
\subfloat[][$\mS=\diagr{N,E,S,W}$]{
{\includegraphics[width=.45\textwidth]{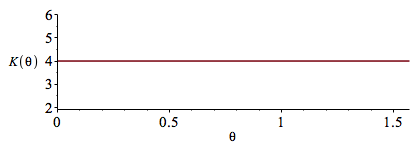}}
}
\qquad	
\subfloat[][Example~\ref{ex:irrat}:~$\mS=\diagr{N,SE,S,SW,W}$]{
{\includegraphics[width=.45\textwidth]{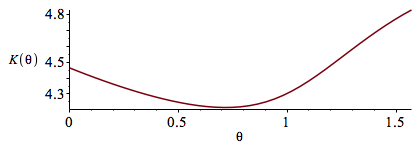}}
}\\	
\subfloat[][Example~\ref{ex:Symmetric}:~$\mS=\diagr{N,SW,S,SE}$]{ 
{\includegraphics[width=.45\textwidth]{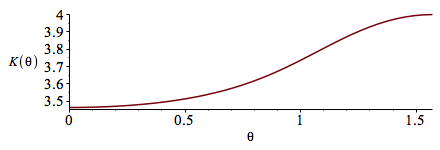}}
}%
\qquad
\subfloat[][$\mS=\diagr{E, N, S, W, SE}$]{ 
{\includegraphics[width=.45\textwidth]{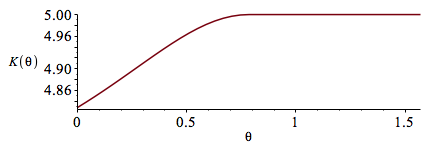}} 
}\\  
\caption{Graphs of $K_\mS(\theta)$ for various $\mS$}	
\label{fig:ub}     
\end{figure}  

Figure~\ref{fig:ub} illustrates this concept by considering several
models, and their exponential growth in several regions. 
\subsection{The main result and the plan of the paper}
Our main result, Theorem~\ref{thm:bestbound}, is an explicitly
computable bound on~$K_\mS$, and its combinatorial interpretation.  It
is determined by minimizing $K_\mS(\theta)$ as a function of $\theta$.
In several cases it is easily seen to be tight, by comparing to the
well-understood subclass of excursions.

We start at Lemma~\ref{lem:hpbij}, where we adapt the formulas of
Banderier and Flajolet to give a formula for~$K_\mS(\theta)$, for
given~$\mS$ and $\theta$. We then show that $K_\mS(\theta)$ defines a
continuous function in~$\theta$. Since each $\mH_\theta$ contains
$\mQ$, $K_{\mS}(\theta)$ is an upper bound on $K_\mS$ for any $\theta$
satisfying $0\leq \theta \leq \pi/2$. Finally, we determine the
location of the minimum upper bound in Theorem~\ref{thm:bestbound} by
basic calculus techniques, since $K_\mS(\theta)$ is an explicit
function of $\theta$.

In Section~\ref{sec:smallsteps}, we show that these results give
precisely the values found by Fayolle and Raschel for the nonsingular
models, demonstrating the bounds are tight. It is
Theorem~\ref{thm:equiv} which vindicates the description of this work
as a combinatorial interpretation of the formulas provided by Fayolle
and Raschel.

Our strategy applies to more general classes of models, for example,
multiple steps in the same direction, longer steps, and higher
dimensional models. The quantities we recover in these cases are,
transparently, upper bounds and they can be compared against
experimental data as a check for tightness. This led us to conjecture
that our approach gives tight upper bounds more generally and hence
actually finds the growth constants, which has subsequently been
proven for some particular cases in~\cite[Corollary 10]{GaRa16}
through probabilistic arguments.
\section{Walks in a half plane}
\label{sec:halfplane}
Models restricted to a half plane are well understood, and we recall
here some basic results. The set~$\mH(\mS)$ of walks restricted to the
upper half plane with steps from the finite multiset~$\mS$ is in
bijection with unidimensional walks with steps from the
multiset~$\mA=\{j: (i,j)\in \mS\}$ because horizontal movement does
not lead to any interaction with the boundary of~$\mH$.  We thus
consider half plane models as unidimensional models defined by sets of
real numbers. We retain the same notation: $\mH(\mA)_n=\{w_0w_1 \dots
w_n: w_i\geq 0, w_i-w_{i-1}\in \mA\}$.  The multiset~$\mA$ is said to
be \emph{nontrivial\/}, if it contains at least one positive and one
negative value.  There are two ways for a multiset~$\mA$ to be trivial
in a half plane: either~$\mA$ contains only non-negative elements and
we call it \emph{unrestricted}; or $\mA$ contains only non-positive
elements.  Unless otherwise stated, we assume that the models are
nontrivial.  It is worth noting that Theorem~\ref{thm:BF} still holds
in the unrestricted case.

The key ingredients for the enumeration are as follows. The
\emph{drift\/} of~$\mA$ is the sum $\delta(\mA) = \sum_{a \in
  \mA} a$ and the \emph{inventory\/} of~$\mA$ is~$A(u)=\sum_{a
  \in \mA}u^a$; notice that these are related by~$\delta(\mA) =
A'(1)$. 

\begin{thm}[Modified from Theorem~4 of Banderier and
  Flajolet~\cite{BaFl02}]
\label{thm:BF}
Let $\mA$ be a multiset of integers which defines a nontrivial
unidimensional walk model. Let~$A(u)=\sum_{a\in\mA} u^{a}$. The
number $h_n$ of walks of length~$n$ in $\mH(\mA)$ depends the
inventory on the sign of the drift $\delta(\mA)=A'(1)$ as follows:

\[
h_n\sim
\begin{cases}
\nu_0\, A(1)^n & \text{if } \delta(\mA)>0\\
 \nu_1\, A(1)^n\,n^{-1/2}  & \text{if } \delta(\mA)=0\\
 \nu_2\, A(\tau)^n\,n^{-3/2}  & \text{if } \delta(\mA)<0\\
\end{cases}.
\]
\noindent Here $\tau$ is the unique positive critical point of $A(u)$ and $\nu_0, \nu_1$ and $\nu_2$ are explicit, real constants.
\end{thm}%

\begin{proof}
  This follows directly from~\cite{BaFl02}. Remark the case of
  unrestricted sets of steps $\mA$, $h_n = |\mA|^n = A(1)^n$ so
  setting $\nu_2 = 1$ gives the result.
\end{proof}

Banderier and Flajolet prove these formulas by applying transfer
theorems to explicit generating functions, which they first
derive. The strategy requires integer steps however: The models with
real-valued steps are not necessarily representable by context-free
grammars, and the generating functions are not necessarily
algebraic. However, results on the growth constant appear in the
probability literature \cite{Done89}. It is rather straightforward to
deduce from the integer case by a limit argument which we present
next.

\begin{thm}
\label{thm:BF-Real}
Let $\mA$ be a multiset of real numbers which defines a nontrivial
unidimensional walk model. Let
$A(u)=\sum_{a\in\mA} u^{a}$. The  number $K_\mA=\lim_{n\rightarrow
  \infty} h_n^{1/n}$, where $h_n$ is of walks of
length~$n$ in $\mH(\mA)$, depends the inventory on the sign of the drift
$\delta(\mA)=A'(1)$ as follows:
\begin{equation}\label{eqn:HALF}
  K_\mA=\begin{cases}|\mA| & \text{if }\delta(\mA)\geq 0\\
    A(\tau) &\text{otherwise}.
\end{cases}
\end{equation}
Here $\tau$ is the unique positive critical point of $A(u)$.
\end{thm}%
First, remark that in the unrestricted case, as $h_n=|\mA|^n$, and the
drift is non-negative, hence the theorem is also true under
weaker hypotheses including this case.

Theorem~\ref{thm:BF} establishes this formula
for~$\mA\subset\mathbb{Z}$. The proof for other real nontrivial models
$\mA \subset \mathbb{R}$ is established from the integer base
case in three steps:
\begin{enumerate}
\item We show that if Equation~\eqref{eqn:HALF} holds for the multiset~$\mA$, then
  it is also true for the multiset~$r \mA=\{ra: a\in
  \mA\}$ when $r > 0$ in Lemma~\ref{thm:scaling};
\item We then deduce that the formula holds for multisets of rationals
  in Remark~\ref{rmk:BFrat};
\item Finally, we prove that the formula holds for multisets of reals by
  proving that a limiting construction of rational models gives the
  result. This is done in Section~\ref{sec:proofof}.
\end{enumerate}

We remark that the growth constant of the sequence counting the number
of excursions of length~$n$ in $\mH$ (in the integer case) can be
shown to be~$A(\tau)$ using a strategy similar to the proof of
Theorem \ref{thm:BF} \cite[Theorem 3]{BaFl02}. This does not translate
as smoothly in the real valued case, as this formula does not
adequately capture when the class is empty.

\subsection{Some facts about the inventory $A(u)$}
The first two steps of the proof of Theorem~\ref{thm:BF-Real} follow
from basic behaviour of~$A(u)$.
\begin{lemma}[Scaling Lemma]\label{thm:scaling}
  Let $\mA$ be a finite multiset of real numbers which is either
  unrestricted or nontrivial and let~$A(u)=\sum_{a\in\mA} u^{a}$.
  Suppose further that Equation~\eqref{eqn:HALF} holds for~$K_\mA$.
  For any~$r>0$, define $\mB = r\mA = \{r a : a \in \mA\}$. Then the
  growth constant $K_\mathcal{B}$ of the sequence
  $b_n=|\mH(\mB)_n|$~satisfies
\[
K_\mathcal{B}=\begin{cases}|\mB| & \text{if }\delta(\mB)\geq 0\\
B(\tau_{\mB}) &\text{otherwise}
\end{cases}.
\]
Here $B(u)=\sum_{b \in \mB} u^b$ and $\tau_\mB$ is the unique
positive critical point of $B(u)$.
\end{lemma}
\begin{proof}
  The lattice model~$\mH(\mB)$ is combinatorially isomorphic
  to~$\mH(\mA)$, so their growth constants are the
  same.  The formula follows because their drifts have the same sign,
  and since~$B(u) = A(u^r)$ and~$\tau_\mB =
  \tau_\mA^{1/r}$, thus~$B(\tau_\mB)=A(\tau_\mA)$.
\end{proof}

\begin{remark}\label{rmk:BFrat}
  For any finite multiset of rational numbers $\mB$, there is an $r>0$,
  for example, the least common multiple of the denominators in $\mB$, so that
  $\mA=r\mB$ is a set of integers.  By Theorem~\ref{thm:BF}, $K_\mA$
  satisfies Equation~\eqref{eqn:HALF}. Consequently, by
  Lemma~\ref{thm:scaling}, since $\mB= \frac{1}{r}\mA$, it is also
  true that $K_\mB$ satisfies Equation~\eqref{eqn:HALF}.
\end{remark}

\begin{lemma}[$A(u)$ is strictly convex at its minimum]\label{thm:chi}
  Given a finite multiset~$\mA$ of real numbers which defines a
  nontrivial unidimensional model, the real valued function~$A(u) =
  \sum_{a\in \mA}u^a$ has a unique positive critical point $\tau$. The
  function is minimized at this point, and is strictly convex on a
  neighbourhood of $\tau$. Furthermore, if $\delta(\mA)=0$, then the
  unique critical point occurs at $u=1$.
\end{lemma}

\begin{proof} 
  If there are no elements in the range $(0,1)$ then the result holds
  term by term.  Otherwise it is possible to scale so that all elements are in that
  range.

\end{proof}

\subsection{The continuity of $A(\tau_{\mA})$ as a function of $\mA$}
To prove Theorem~\ref{thm:BF-Real}, we consider the
function~$A(u)$, evaluated at its critical point. In particular we
view this as a function of the step lengths. In the following
lemma,~$\ell$ represents the number of elements in the step set.
\begin{lemma}\label{lem:smooth}
Let $F : \mathbb{R}^\ell\times \mathbb{R}_{> 0} \rightarrow
\mathbb{R}$ 
be the function defined by
\[ F( ( x_1, x_2,\dots, x_\ell), u) = F(\vecx,u) = \sum_{j=1}^\ell
u^{x_j}.\] Furthermore, let~$\veca\in\mathbb{R}^\ell$ have at least one
positive component and one negative component,  and denote
by~$\tau(\veca)$ the unique positive critical point of the map
$u \mapsto F(\veca,u)$.  There is a neighbourhood
$\mU$ of $\veca$ such that the function
\[
\kappa(\vecx)=F(\vecx, \tau(\vecx)),
\]
is continuous in $\vecx$ on~$\mU$.
\end{lemma}
\begin{proof}
  The results is a consequence of the implicit function theorem applied to $f(\vecx,u)=\frac{\partial F(\vecx, u)}{\partial u}$ around the
  point $\veca$.
%
%
\end{proof}

\subsection{Proof of Theorem~\ref{thm:BF-Real} in the case of real
  multisets}
\label{sec:proofof}
\begin{proof}[Proof of Theorem~\ref{thm:BF-Real}]
  Let $\mA \subset \mathbb{R}$ be a finite multiset of real numbers which is either
  unrestricted or nontrivial. To prove that
  $K_\mA$ satisfies Equation~\eqref{eqn:HALF} we build
  two sequences of rational step sets which converge to $\mA$. We then
  squeeze the growth constant~$K_\mA$ of~$h_n = |\mH(\mA)_n|$ into the
  desired form. 

  For each $a \in \mA$, let $\{a_i^+\}$
  and $\{a_i^-\}$ be rational sequences satisfying 
\[
0 \leq a_i^+ - a \leq \frac{1}{2^{i}} \qquad \mbox{and}\qquad  0 \leq a - a_i^- \leq \frac{1}{2^{i}}.
\]
We define two multisets
\[
\mA_i^+ =  \{ a_i^+ : a \in \mA \}\qquad\mbox{and} \qquad
\mA_i^- =  \{ a_i^- : a \in \mA \}.
\]
The drift is additive, thus for each $i$,
\[
\delta(\mA_i^-)=\sum_{a\in\mA} a_i^- \leq \sum_{a\in\mA} a\leq
\sum_{a\in\mA} a_i^+ =\delta(\mA_i^+).
\]

Note that Remark~\ref{rmk:BFrat} applies to both
$\mA_i^-$, and $\mA_i^+$ because both are multisets of rational numbers and
hence \eqref{eqn:HALF} is valid for the growth constants,
$K_{\mA_i^-}$ and $K_{\mA_i^+}$ respectively.

Furthermore, starting from a half-plane walk and making some steps slightly more positive can never talk the walk outside the half-plane so, by the construction of $\mA_i^+$ and $\mA_i^-$ we get a natural injection
\[
\mH(\mA_i^-) \xhookrightarrow \mH(\mA) \xhookrightarrow \mH(\mA_i^+)
\] 
and hence
\begin{equation}\label{eqn:k}
K_{\mA_i^-} \leq K_{\mA} \leq K_{\mA_i^+}.
\end{equation}
We claim
\[
\lim_{i \rightarrow \infty} K_{\mA_i^-} = K_\mA = \lim_{i \rightarrow \infty} K_{\mA_i^+},
\]
and $K_\mA$ is given by the formula of Theorem~\ref{thm:BF}. To prove this claim
observe the following.
In the negative drift case Lemma~\ref{lem:smooth} guarantees that the $ \lim_{i \rightarrow \infty} K_{\mA_i^\pm}$ converge and, from Lemma~\ref{thm:chi}, when $\delta(\mA) = 0$, $\tau_{\mA} = 1$ so a negative drift $\mA_i^-$ limits correctly up to a $0$ drift $\mA$.  In all cases $\mA_i^{\pm} \rightarrow \mA$ so $\delta(\mA_i^{\pm}) \rightarrow \delta(\mA)$.  Therefore by Remark~\ref{rmk:BFrat}
\[
 \lim_{i \rightarrow \infty} K_{\mA_i^\pm} = \begin{cases}|\mA| & \text{if }\delta(\mA)\geq 0\\
A(\tau_{\mA}) &\text{otherwise}\end{cases}
\]
so the squeeze theorem gives the result.
\end{proof}

\subsection{Other half planes}
\label{sec:otherhp}
Next, we extend to other half planes, each defined by an angle:
$\mH_\theta=\{(x,y): x\sin\theta+y\cos\theta\geq 0\}$. Note that for
$\theta \in [0,\pi/2)$ this region is equal to $\{(x,y):y\geq -mx \}$
where $m=\tan\theta$.  The upper half plane is given by $\mH_0$ and
the right half plane is $\mH_{\pi/2}$. In this latter case, we use the
extended reals, and write $m=\infty$.  The enumeration of lattice paths in~$\mH_\theta$ emulates the enumeration of lattice paths in~$\mH$.

\begin{figure}
  \subfloat[The step set $\mS=\{N, E, SW\}$ (in
  blue), and its projection onto the line $y=x/2$ (in purple, dashed) ]{
\begin{minipage}[t]{.45\textwidth}\center
\begin{tikzpicture}[scale=1.5]%
                \fill[fill=gray!10, rotate=297]  (-1.4,0) rectangle (1.3,1.3);
                \draw [step=1,thin,gray!30] (-1.5,-1.5) grid (1.1,1.1);
                \draw [color=black] (0,1.1) -- (0,-1.5);
                \draw [color=black] (1.1,0) -- (-1.5,0); 

                \draw [color=gray!60,line width=0.5mm] (-1.2,-0.6)--(-1,-1);  
                \draw [color=gray!60,line width=0.5mm] (0.4,0.2)--(0,1);
                \draw [color=gray!60,line width=0.5mm] (0.8,0.4)--(1,0);
                \draw [->, color=blue!80!gray,line width=1mm](0,0)--(0,1);
                 \draw [->, color=blue!60!black,line width=1mm](0,0)--(1,0);
                \draw [->, color=blue!80!white,line width=1mm] (0,0)--(-1,-1);

                \draw [color=red!80, thick] (-0.6,1.2) -- (0.6,-1.2);
                \draw [color=red!50, thick] (-1.6,-0.8) -- (1.2, 0.6);
                
                \draw [dashed,->,color=violet!80!black,line width=0.85mm] (0,0) -- (0.8, 0.4);
                \draw [dashed, ->,color=violet!80!red,line width=0.85mm] (0,0) -- (0.4, 0.2);
                \draw [dashed,->,color=violet!80!white,line width=0.85mm] (0,0) -- (-1.2, -0.6);
              \end{tikzpicture}

\end{minipage}}
\subfloat[Step set for the unidimensional model~$H(\{1,2,-3\})$]{
\begin{minipage}[t]{.45\textwidth}\center
\begin{tikzpicture}[scale=.75]%
	        \fill[fill=gray!10]  (-1.4,0) rectangle (3.4,2.5);
                \draw [step=1,thin,gray!30] (-1.4,-3.4) grid (3.4,2.4);
                \draw [color=red!50, thick] (0,-3.4) -- (0,2.4);

                \draw [->,color=violet!80!black,line width=1.75mm] (0,0) -- (0,2);
                \draw [->,color=violet!80!red,line width=1.75mm] (0,0) -- (0,1);
                \draw [->,color=violet!80!white,line width=1.75mm] (0,0) --
                (0,-3);
                \draw [color=red, thick] (-1.4,0) -- (3.4,0); 
 
\end{tikzpicture}

\end{minipage}}
\caption{Three representations for models of walks with steps from
  $\mS=\{N, E, SW\}$ restricted to the region~$\{y\geq -2x \}$,
  defined by $\theta=\arctan(2)$:
 {\sc (A)} $H_{\theta}(\mS)$, and its unidimensional projection and
 {\sc (B)}
 $H(\{1,2,3\})$, a scaling of $\mA(\theta)$.}

\label{fig:Smmodel}

\smallskip

{\color{gray}\hrule}
\end{figure}
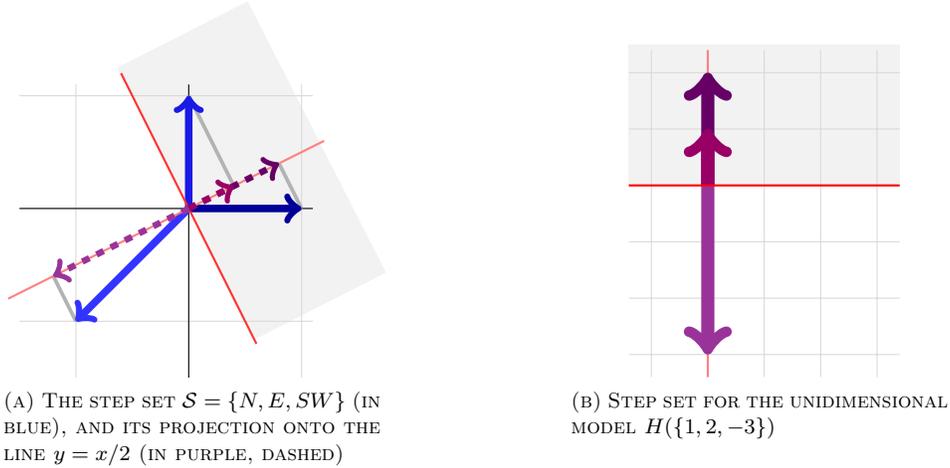
\begin{lemma}
\label{lem:hpbij} 
Let $\mS\subset\mathbb{Z}^2$ be a finite multiset and let
$\mH_\theta=\{(x,y): x\sin\theta+y\cos\theta\geq 0\}$ and
let~$\mA(\theta)=\{i\sin\theta +j\cos\theta: (i,j)\in\mS\}$. The
combinatorial class $\mH_{\theta}(\mS) $ is combinatorially
isomorphic to~$\mH(\mA(\theta))$.

Furthermore, if $\mS$ is a nontrivial or unrestricted step set for
$\mH_\theta$ then the growth constant $K_\mS(\theta)$ for the sequence
$|\mH_\theta(\mS)_n|$ is the value $K_{\mA(\theta)}$ determined by
Theorem~\ref{thm:BF-Real}.
\end{lemma}

\begin{proof}
  Here it suffices to consider the displacement of each step in the
  step set in the direction orthogonal to the boundary. See
  Figure~\ref{fig:Smmodel} for an example. The steps $(0,1)$ and
  $(1,0)$ respectively have displacement $\cos\theta$ and $\sin\theta$
  in this direction, the other steps follow by linearity.
This gives rise to a unidimensional
  half plane model with step set $\mA(\theta)$ to which
  Theorem~\ref{thm:BF-Real} applies if $\mA(\theta)$ is nontrivial or
  unrestricted, which occurs precisely when~$\mS$ is nontrivial or unrestricted in~$\mH_\theta$.
\end{proof}

\begin{example}[$\mS=\{N, E, SW\}$=\diagr{N, E, SW}] For
  any~$\theta\in[0,\pi/2]$, $\mQ(\mS)\subseteq \mH_\theta(\mS)$ and
  the following classes are  combinatorially isomorphic
  \[  \mH_\theta(\mS)\cong \mH(\{\cos\theta, \sin\theta,
  -\cos\theta-\sin\theta \}).\] When $\theta\neq\pi/2$, we can scale
  the model by~$\cos\theta^{-1}$. Let $m=\tan\theta, \theta\neq\pi/2$, then $\mH_\theta(\mS)\cong
  \mH(\{1, m, -m-1\})$. For $\theta=\pi/2$, remark~$\mH_{\pi/2}(\mS)\cong\mH(\{1,0,-1 \})$.
\end{example}

\section{Bounds for lattice path models in the quarter plane}
\label{sec:lattice-intro}
As we have already noted in the introduction, the exact enumeration of
quarter plane models has been well explored recently. In the case of
small steps, Bousquet-M\'elou and Mishna identified 79 non-isomorphic,
nontrivial small step models~\cite{BoMi10}. The associated generating
functions are known to be D-finite\footnote{The generating functions
  satisfy linear differential equations with polynomial coefficients.}
for 23 of these models.  After algebraic, the D-finite models are
easiest to enumerate asymptotically, and several approaches for this
have been successful~\cite{BoKa09, MeMi16, MeWi16, BoChHoKaPe16}.

The non-D-finite models have been more elusive. Fayolle and Raschel
have determined expressions for the growth constant for 74
models~\cite{FaRa12}. We summarize their formulas in
Section~\ref{sec:smallsteps}. Melczer and Mishna determined formulas
for the the remaining five models~\cite{MeMi14a}, and also for highly
symmetric models of arbitrary dimension~\cite{MeMi16}.

To describe these formulas we again need the drift of the model,
denoted $\delta(\mS):$
\[
\delta(\mS) = \sum_{s\in\mS}s=(\delta_x,\delta_y).
\]
Here we use a shorthand for classifying drift profiles. For each
component, we note if the drift is positive $(+)$, zero $(0)$ or
negative $(-)$. For example, if $\delta(\mS) \in \mathbb{R}_{>0}\times
\mathbb{R}_{\geq 0}$,
the drift profile is $(+,+/0)$.

The \emph{inventory\/} of the model is the Laurent polynomial $S(x,y)$
defined
\[
S(x,y)=\sum_{(i,j)\in\mS} x^i y^j.
\]
This is the two dimensional analog to~$A$, and it can be used to
express some useful quantities
\[
S(1,1)=|\mS|\qquad \delta_x = \left.\frac{\partial}{\partial x}S(x,1)\right|_{x=1}=P_x(1,1), 
\qquad \delta_y =  \left.\frac{\partial}{\partial y}S(1,y)\right|_{y=1}=P_y(1,1).
\]

The unidimensional case analysis depended upon the existence of a
positive critical point of the inventory. Its existence was a
consequence of the non-triviality of the model. The two dimensional
case is similar.

We use the result that in the case of a nontrivial, non singular model
there is a unique solution~$(\alpha, \beta) \in \mathbb{R}^2_{>0}$ to
the equation $P_x(x,y)=P_y(x,y)$ when~$\mS$ is nonsingular. 
This is a
straightforward consequence of equation manipulation. 
 We call this point the
\emph{critical point\/} of the inventory. It is straightforward to
show that there is no such~$(\alpha, \beta)$ with~$\alpha$ and~$\beta$
positive when~$\mS$ is singular.

\subsection{Bounds from half-plane models}
\label{sec:upperbound}
An upper bound on the growth constant of a quarter plane model can
always be determined by appealing to a half plane model using the same
steps restricted to lie in a region containing the first quadrant. In
this section we describe how to determine the half plane which gives
the best bound. The main result is Theorem~\ref{thm:bestbound}. It is
followed by examples of its application, its proof, and then in
Section~\ref{sec:smallsteps} a proof that, in the case of small steps,
the bound is the same as the exact formula of Fayolle and Raschel. In
some cases, this is easy to see, as the upper bound is the same as the
lower bound given by excursions. 

\begin{thm}[Main Theorem]
\label{thm:bestbound}
Let $\mS \subset \mathbb{Z}^2$ be a finite multiset that defines a
nontrivial quarter plane model~$\mQ(\mS)$.  Then, 
\begin{enumerate}
\item  the growth constant~$K_\mS=\lim_{n\rightarrow\infty} q_n^{1/n}$ satisfies
   \begin{equation*}
      K_\mS \leq K_\mS(\theta)\quad \text{ for all } \quad0\leq \theta
       \leq \pi/2,
    \end{equation*}
where $K_\mS(\theta)$ is the growth constant for the associated
rotated half plane model, as defined in Lemma~\ref{lem:hpbij};
\item the function $K_\mS(\theta)$ is continuous
as a function of $\theta$. 
\end{enumerate}
If~$\mS$ is non-singular, with inventory $S(x,y)$, 
then denote by
  $(\alpha,\beta)$ the unique solution in $\mathbb{R}^2_{> 0}$~to
\[
P_x(\alpha, \beta)=P_y(\alpha,\beta)=0.
\]
If~$\beta \neq 1$ and $\ln\alpha/\ln\beta\geq 0$, then let
$\theta^*=\arctan\frac{\ln \alpha}{\ln \beta}$. If
$\delta_x\sin\theta^*+\delta_y\cos\theta^* \leq 0$ then the minimum
value of $K_\mS(\theta)$ is attained at $\theta^*$,
and \[K_\mS(\theta^*) =S(\alpha, \beta).\] Otherwise, if any of
these conditions are not satisfied, or if $\mS$ is singular, the
minimum is attained at one of the endpoints of the range: either $0$
or $\pi/2$.
\end{thm}

Note that the minimum being obtained as described does not preclude it
also being attained elsewhere. Notably, $K_\mS(\theta)$ is constant in
the range $0\leq \theta \leq \pi/2$ when the drift profile is
non-negative in each component. We also note that the drift condition
$\delta_x\sin\theta^*+\delta_y\cos\theta^* \leq 0$ only comes into
play for the mixed drift profiles $(+, -)$ and $(-, +)$.  The basic idea is that $K_\mS(\theta)$ gets truncated at $|\mS|$ but otherwise its critical point behavious comes from that of $S(x,y)$ in the manner detailed below.  So the minimum of $K_\mS(\theta)$ comes either from $S(\alpha, \beta)$, from an endpoint, or from the constant truncated part, which then must also be achieved at an endpoint.

Before we prove this result, we consider three examples to develop some
intuition on the behaviour. Recall that to each pair of quarter plane
model $\mS$, and angle $\theta$, we associate the unidimensional step
set $\mA(\theta)=\{i\sin\theta +j\cos\theta: (i,j)\in\mS\}$. We define
the following shorthand for its inventory and critical point:
\[A_\theta \equiv \sum{a\in_{\mA(\theta)}} u^a\quad\text{and}\quad \tau_\theta \equiv
\tau_{\mA(\theta)}.
\]

\begin{example}[$\mS = \diagr{NW, N, NE, SE}$]\label{ex:singular} 
  This is an example of a nontrivial, singular model. There is no
  point $(\alpha, \beta)$ as in the theorem statement. The drift
  profile is $(+,+)$, and hence $K_\mS(\theta)$ is constant in the
  domain. We deduce $K_\mS \leq K_\mS(\theta)=|S|=4$. This value is
  tight, according to the formulas of Melczer and
  Mishna~\cite{MeMi14a}.
\end{example}

\begin{example}[$\mS=\diagr{N, SW, S, SE}$]\label{ex:Symmetric}  This is not a singular
  model, and the inventory
  $S(x,y)=y+\frac{1}{y}+\frac{x}{y}+\frac{1}{xy}$ has a unique
  critical point at $(1, \sqrt{3})$. The optimal angle given by
  Theorem~\ref{thm:bestbound}
  is\[\theta^*=\arctan(\ln(1)/\ln(\sqrt{3}))=0.\]

  Now, $\mQ(\mS)\subset\mH(\mS)$, and $\mH(\mS)\cong
  \mH(\mA(0))\cong\mH(\{1,-1,-1,-1\})$. This leads to the
  bound~$K_\mS\leq K_\mS(0)=A_0(\sqrt{3}) =S(1,\sqrt{3})=
  2\sqrt{3}$. This is tight in comparison to the formula given for
  $K_\mS$ by Fayolle and Raschel~\cite{FaRa12}. We shall see that this
  is a special case of Lemma~\ref{lem:rhoX}.
\end{example}

\begin{example}[$\mS = \diagr{N,W,SE,S,SW}$]\label{ex:irrat} We can
  numerically compute the critical point $(\alpha, \beta)$ of the
  inventory~$S(x,y) = y + \frac{x}{y} + \frac{1}{y} + \frac{1}{xy} +
  \frac{1}{x}$ as $(\alpha,\beta) \approx (1.6760,
  1.8091)$. Consequently, the minimising angle is $\theta^* \approx
  0.2281\pi \approx \arctan(0.8712)$.  The best bound is
  computed~$K_\mS(\theta^*) \approx 4.2148$. The exact value of the
  critical point is a tight bound. Note that the minimising half
  plane is not defined by the line perpendicular to the drift. The
  drift vector is $\delta(\mS) = (-1,-2)$, so the perpendicular has
  slope $-1/2$. We contrast this with slope given by the bound, at
  $-0.8712$.
\end{example}

Example~\ref{ex:irrat} demonstrates that the best half plane is not
defined by the perpendicular to the drift vector (a common
hypothesis). Rather, the slope is connected to the Cram\'er
transformation in probability~\cite{DeWa15}. Denisov and Wachtel
assign the probability
$p_{ij}=\frac{\alpha^i\beta^j}{S(\alpha,\beta)}$ to the step $(i,j)$
so that the drift of the weighted steps, given by
$\sum_{(i,j)\in\mS} p_{ij}$, is $(0,0)$, and then apply tools for
walks with no drift. It is clear here, perhaps, why their methods
apply only to nonsingular walks -- they require the existence of
$\alpha$ and $\beta$. Bostan, Raschel and Salvy in~\cite{BoRaSa14}
also discuss the combinatorics of this transform, and show that
~$S(\alpha, \beta)$ is the growth constant for excursions in the
quarter plane, which is a subclass of walks, hence $S(\alpha, \beta)$
is a lower bound for the growth constant. We offer the interpretation
of~$S(\alpha, \beta)$ as the exponential growth of walks restricted to
the half plane $H_\theta$ for the angle
$\theta = \arctan(\frac{\ln\alpha}{\ln\beta})$ (or the right half
plane when $\beta=1$).

\subsection{The function $K_\mS(\theta)$}
\label{sec:bestbound}
The function $K_\mS(\theta)$ is surprisingly simple, for fixed $\mS$.
In the case of half-plane walks, the growth constant for the counting
sequence for the walks with arbitrary endpoint is either the number of
steps, or given by the growth constant for excursions. The deciding
factor is the drift. 

In our model of changing half planes, the drift is given by the following smooth
function of $\theta$:
\[
\sum_{(i,j) \in \mS} (i\sin\theta + j\cos\theta) = \delta_x\sin\theta+\delta_y\cos\theta.
\]
Thus, $K_\mS(\theta)$ is either the number of steps, or given by
$A_\theta(\tau_\theta)$, and switches between them when the drift
is 0.  Several different possibilities are presented in
Figure~\ref{fig:ub}.  Roughly, the function $A_\theta(\tau_\theta)$
is $\pi$-periodic and attains a single maximum given by the number of
steps, and a single minimum. If that minimum is in the
interval~$[0, \pi/2]$ it is also the minimum of $K_\mS(\theta)$.

These functions are well behaved, and we can accurately predict the
point where $K_\mS(\theta)$ attains a minimum. First, we establish the
continuity of $K_\mS(\theta)$ in the next lemma, and then we determine
the complete set of critical points, and whether or not they are
maxima, or minima, in Lemma~\ref{lem:critpoints}.

\begin{lemma}\label{lem:continuity}
  Suppose~$\mS\subset\mathbb{Z}^2$ defines a nontrivial quarter plane
  model. Then $K_\mS(\theta)$ defines a continuous function on the
  domain $\theta\in[0, \pi/2]$.
\end{lemma}
\begin{proof}
The value of $K_\mS(\theta)$ is defined piecewise according to the
value of $\delta(\mA(\theta)) = \sum_{(i,j) \in \mS} (i\sin\theta + j\cos\theta)
= \delta_x\sin\theta+\delta_y\cos\theta$:
\[
K_\mS(\theta)=
\begin{cases} 
|\mS| & \text{if } \delta(\mA(\theta)) \geq 0\\
A_\theta(\tau_\theta) & \text{otherwise}.
\end{cases}
\]

The function $A_\theta(\tau_\theta)$ is continuous as a consequence of Lemma~\ref{lem:smooth} since $A_\theta(\tau_\theta) = \kappa(x(\mathcal{S}, \theta))$ where \[\kappa(\vecx)=\sum_{j=1}^\ell \tau(\vecx)^{x_j}.\]
The condition $\delta(\mA(\theta)) \geq 0$ defines an interval in $[0, \pi/2]$; consequently $K_\mS(\theta)$ is piecewise continuous.
Finally, as
$\delta(\mA(\theta))$ approaches $0$, $K_\mS(\theta)$ tends to
$|\mS|$, by Lemma~\ref{thm:chi} the function is continuous at points where
$\delta(\mA(\theta))=0$.  

\end{proof}

Next we pinpoint the location of the minima. 
\begin{lemma}\label{lem:critpoints}
  Let $\mS\subseteq\mathbb{Z}^2$ be the step set of a nontrivial
  quarterplane model.  When they exist let $(\alpha, \beta)$ be the 
  positive critical point of $S(x,y)$ and $\theta'$ be 
  $\arctan(\ln\alpha/\ln \beta)$.  If they exist and 
  $\delta(\mA(\theta')) \leq 0$ then $K_\mS(\theta)$ achieves its minimum
  value at $\theta'$.  Otherwise $K_\mS(\theta)$ achieves its minimum 
  value at an end point.

Specifically, the minimum value of~$K_\mS(\theta)$
  is achieved at $\theta^*$, determined as follows
\begin{enumerate}
\item if  $\delta(\mS)=(0/+, 0/+)$, then for all $\theta^*\in
  [0, \pi/2]$,  $K_\mS(\theta^*)=|S|$;
\item if~$\mS$ is singular then, $\theta^*\in \{0, \pi/2\}$;
\item if~$\mS$ is nonsingular then the inventory $S(x,y)$ has a unique
  critical point~$(\alpha, \beta) \in\mathbb{R}_{\geq 0}^2$.  If
  $\beta=1$ then $\theta^*=\pi/2$. If not, set
  $\theta'= \arctan(\ln \alpha/\ln \beta)$. If
  $\delta(\mA(\theta'))\geq 0$, then $\theta^*\in \{\pi/2,
  0\}$.
  Otherwise, $\theta^*=\theta'$.  In this final case, the growth
  constant coincides with that of excursions, that is,
  $K_\mS(\theta^*) =S(\alpha, \beta)$.
\end{enumerate}
\end{lemma}
\begin{proof}
The function is decided by the evolution of the drift. If the drift
profile is $(+/0, +/0)$, then   $K_\mS(\theta)=|S|$. Otherwise, 
the curve takes the value of $A_\theta(\tau_\theta)$ for some
sub-interval, and it is in this interval where the minimum occurs.  


It turns out that it is easier to work with the value that determines
the slope, $m=-\tan(\theta)$. To this end we define
$\overline{\mA}(m)=\{ im+j : (i,j)\in\mS \}$, and
$\overline{\mA}(\infty)=\{i: (i,j)\in\mS\}$. Remark that
$\overline{\mA}(m)$ is a scaled version of $\mA(\theta)$, scaled by
$\cos\theta^{-1}$. Thus, the two unidimensional models are are
combinatorially isomorphic and so the exponential growth of the two
models are the same. We find the minimizing slope, $m^*$.

Consider $G(u,m)=\sum_{(i,j)\in\mathcal{S}}u^{im+j}$, the inventory of
$\overline{A}(m)$. Now, $K_\mS(\theta)=G(u_m^*,m)$ where $u_m^*>0$
satisfies $G_u(u_m^*, m)=0$. We minimize $G(u_{m^*}, m)$ as a function of~$m\geq 0$
by solving for $(u,m)$ satisfying $G_m(u,m)=0$ and $G_u(u,m)=0$.  We
apply the chain rule, by first remarking $G(u,m)=S(u^m, u)$: 
\begin{eqnarray}
G_m(u,m)& =& \left. u^{m}\ln u P_x(x,y)
\right|_{\substack{x=u^m\quad\\y=u}}\label{eq:gm}\\
G_u(u,m)& =&
\left. mu^{m-1} P_x(x,y)\right|_{\substack{x=u^m\quad\\y=u}}
+\left.P_y(x,y)\right|_{\substack{x=u^m\quad\\y=u}}. \label{eq:gu}
\end{eqnarray}
There is a solution when~$u=1$ and~$G_u(u,m)=0$.  This is precisely
the case when $-\delta_x/\delta_y= m$, since $P_x(1,1)=\delta_x$ and
$P_y(1,1)=\delta_y$. In Lemma~\ref{lem:continuity} we have shown this
to be a maximum value, since~$G(1, m)=|\mS|$, the largest possible
value, at this point.

The only other possible solution is when $P_y(x,y)=P_x(x,y)=0$ itself has positive solution $(\alpha, \beta)$.  Thus either the minimum of $K_\mS(\theta)$ comes from such an $(\alpha, \beta)$ or it is at an end point.  It remains only to determine when each occurs.

If $\mS$ is singular, no appropriate $(\alpha, \beta)$ exists so the
minimum of $A_\theta(\tau_\theta)$, and hence $K_\mS(\theta)$, occurs at
a boundary.

If~$\mS$ is nonsingular, there is a unique positive $(\alpha,
\beta)$. 
If $\beta=1$, the minimum occurs when
$\theta=\pi/2$, so assume that $\beta\neq 1$. Then $(u^*, m^*)=(\beta,
\ln\alpha/\ln\beta)$ is a critical point of $G(u,m)$ in the desired
domain.
Now, it is possible that when this point occurs, in fact, the
half-plane model defined by $\theta$ is in a positive drift regime.
In this case the minimum of $K_\mS(\theta)$ is at a boundary.  Assume
otherwise that the half-plane models near the critical point have
negative drift.

In this cse, $A_\theta(\tau_\theta)$ has a minimum at the angle $\theta^*$ corresponding to $m^*$.  For any fixed $m$
near $m^*$, the function $G(u,m)$ is convex as a function of $u$. For fixed $m$ near $m^*$, the minimum of $G(u,m)$ as a
function of $u$ is, by definition, $A_\theta(\tau_\theta)$ for the angle
corresponding to $m$.  Consequently in a sufficiently small
neighbourhood of $(u^*, m^*)$, $G(u,m)$ is convex both in $u$ and in
$m$, so $G(u,m)$ is convex as a two variable function at $(u^*, m^*)$.
Therefore this $(u^*, m^*)$ is a minimum.

Therefore when the half-plane drift at the critical point is negative,
the minimum of both $A_\theta(\tau_\theta)$ and $K_\mS(\theta)$
occurs at $\theta^*$.
In this case, we compute the
growth factor by evaluating at the critical point. That is,
\[
K_\mS(\theta^*)= G(u^*, m^*) = S((u^*)^{m^*}, u^*) = S(\alpha, \beta).
\] 
\end{proof}
Now, we put these ideas together.
\begin{proof}[Proof of Theorem~\ref{thm:bestbound}]
  Since $\mS$ is a nontrivial model, then $A_\theta$ is either a
  nontrivial or an unrestricted model for all
  $\theta\in[0,\pi/2]$. For any $\theta\in[0,\pi/2]$, $q_n \leq
  |\mH{(\mA(\theta))}_n|$ for any $n$, and thus the growth constants
  of the sequence satisfy $K_\mS\leq K_\mS(\theta)$.  The latter is
  minimized at the stated $\theta^*$ by~Lemma~\ref{lem:critpoints}.
\end{proof}

\section{The case of small steps: a different approach}
\label{sec:smallsteps}
\subsection{The work of Fayolle and Raschel}
Fayolle and Raschel~\cite[Remark 4.9]{FaRa12} describe the location of
the dominant singularity in the generating function for nonsingular
small step quarter plane models. Their formula depends on the drift
$\delta(\mS) = (\delta_x, \delta_y)$, along with another parameter of
the model called the \emph{covariance}.  We do not use this parameter,
except to compare to their formulas. The covariance of a step set,
denoted $\gamma(\mS)$ is defined as
\[
\gamma(\mS) =  \left.\frac{\partial^2}{\partial x\partial y}S(x,y)\right|_{(x,y)=(1,1)}- \delta_x\delta_y.
\]
In the case of small steps the inventory
always has the form
\[
   S(x,y) = a(x)y + b(x) + c(x)y^{-1} = \tilde{a}(y)x + \tilde{b}(y) + \tilde{c}(y)x^{-1}.
\]
They prove that there are four possible values for~$K_\mS$:
\begin{equation}\label{eqn:rhoX}
\begin{gathered}
|\mS|,\qquad \rho_0^{-1}\equiv S(\alpha,\beta)\\
\rho_Y^{-1}\equiv {b(1) + 2\sqrt{a(1)c(1)}},\qquad 
\rho_X^{-1}\equiv {\tilde{b}(1) +2\sqrt{\tilde{a}(1)\tilde{c}(1)}}.
\end{gathered}
\end{equation}
As before, $(\alpha,\beta)$ is the unique positive solution in
$\mathbb{R}_{>0}^2$ satisfying
\[
 P_x(\alpha,\beta) = 0 \qquad P_y(\alpha,\beta) = 0.
\]
Furthermore, in their Remark 4.9 they determine conditions on the
sign of
~$\delta(\mS)$ and $\gamma$ which decide which of the four values is
correct for a given nonsingular model. 

These results have some natural interpretations. If $\delta(\mS)$ is
non-negative in both components then the growth constant is as for
unrestricted walks.  If $\delta(\mS)$ is positive in the first
component, and negative in the second, then growth constant is the
same as the walks that remain in the upper half plane.  The value
$S(\alpha, \beta)$ is the growth constant for excursions~\cite{DeWa15,
  BoRaSa14}. If $\delta(\mS)$ is negative in both components then the
growth constant is the same as the growth constant for excursions in
the region. This mirrors the behaviour in the case of unidimensional
walks.  The specific formulas they obtain are simply the result of their cases picking out when $K_\mS(\theta)$ is minimized at an end point or the critial point.
In this way, we are able \emph{unify the six cases}, and deliver \emph{a single
  interpretation\/} of the formulas. 

Specifically, the next lemma shows how the values
$\rho_0,\rho_X,\rho_Y,1/|\mS|$ arise as $K_\mS(\theta)$, for various
$\theta$.
\begin{lemma}\label{lem:rhoX}
For any nontrivial quarter plane model $\mQ(\mS)$ with
$\mS\subset\{0,\pm1\}^2$, the following equalities hold:
\[
\rho_Y^{-1}=A_0(\tau_0)\qquad \rho_X^{-1}=A_{\pi/2}(\tau_{\pi/2}).
\]
\end{lemma}
\begin{proof}
This is proved by simply unraveling the notation: 
\begin{equation*}
A_0(u)=\sum_{(i,j)\in\mS} u^j = S(1,u)
\end{equation*}
\[\text{and so}\quad A_0'(u)=[u]S(1,u)-\frac1{u^2}[u^{-1}]S(1,u) \implies
\tau_0=\sqrt{\frac{[u^{-1}]S(1,u)}{[u]S(1,u)}}.\]
Consequently,
\[A_0(\tau_0)=[y^0]S(1,y)+2\sqrt{[y]S(1,y)\cdot
  [y^{-1}]S(1,y)}=\rho_Y^{-1},\] which is precisely the formula for
$\rho_Y^{-1}$ given in Equation~\ref{eqn:rhoX}.
 The $\rho_X$ case is
similar since $A_{\pi/2}(u)=\sum_{(i,j)\in\mS}u^i$.
\end{proof}

We conclude this section with a discussion that the bounds are tight,
i.e.~\[K_\mS = \min_{\theta\in[0,\pi/2]} K_\mS(\theta)\]
for all small step quarter plane models. Recall
that subsequent to the first version of this document being circulated,
 this has been proved by Garbit and Raschel, but we can understand it
combinatorially.
\begin{table}
\begin{tabular}{|c|c||c|c|c|c||c|}\hline
$(\delta_x,\delta_y)$ &$\gamma$& $\alpha$ & $\beta$
&$\ln(\alpha)/\ln(\beta)$&$\tan(\theta^*)$ & $K_\mS$\\\hline\hline
$(+,+)$ && \multirow{3}{*}{$<1$}& \multirow{3}{*}{$<1$} & \multirow{3}{*}{$+$} & \multirow{3}{*}{$\frac{\ln\alpha}{\ln\beta}$} &\multirow{3}{*}{$|\mS|$}\\
$(+,0)$ &&&&&&\\
$(0, +)$&&&&&&\\\hline\hline
$(0,0)$&& $1$& $1$ &  & 1 & $S(\alpha,\beta)=|\mS|$\\\hline\hline
\multirow{3}{*}{$(0,-)$} &$-$&$>1$&\multirow{4}{*}{$>1$}&$+$&$\frac{\ln\alpha}{\ln\beta}$&$S(\alpha,\beta)$\\\cline{2-3}\cline{5-7}
 &$0$&$1$&&$0$&$0$&$S(\alpha, \beta)=\rho_Y^{-1}$\\\cline{2-3}\cline{5-7}
 &$+$&$<1$&&$-$& $0$ &\multirow{2}{*}{$\rho_Y^{-1}$}\\\cline{1-3}\cline{5-6}
$(+,-)$ &&$<1$&&$-$&$0$ &  \\\hline\hline
\multirow{3}{*}{$(-,0)$} &$-$&\multirow{4}{*}{$>1$}&$>1$&$+$&$\frac{\ln\alpha}{\ln\beta}$&$S(\alpha, \beta)$\\\cline{2-2}\cline{4-7}
 &$0$&&$1$&$\infty$&$\infty$&$S(\alpha, \beta)=\rho_X^{-1}$\\\cline{2-2}\cline{4-7}
 &$-$&&$<1$&$-$&$\infty$&\multirow{2}{*}{$\rho_X^{-1}$}\\\cline{1-2}\cline{4-6}
$(-,+)$ &&&$>1$&$-$&$\infty$& \\\hline\hline
$(-,-)$ && $>1$ & $>1$& + & $\frac{\ln\alpha}{\ln\beta}$& $S(\alpha, \beta)$\\\hline
\end{tabular}\\[3mm]
\caption{The value of $K_\mS$ for nonsingular nontrivial quarter plane
  models defined by a finite, step set~$\mS \subset\{\pm1, 0\}^2$. This value can be
  computed either using sign information about the drift~$\delta$ and
  the covariance~$\gamma$, or from information about~$\alpha$ and~$\beta$. 
}
\label{tab:minvals}
\end{table}

\begin{thm}\label{thm:equiv}
  Let $\mS\subseteq\{0,\pm1\}^2$ be a finite set defining a
  nontrivial, nonsingular quarter-plane lattice path model. The growth
  constant~$K_\mS$ for the number $q_n$ of walks of length $n$ in
  $\mQ(\mS)$ satisfies
\begin{equation}
K_\mS=\min_{\theta\in[0,\pi/2]} K_\mS(\theta).
\end{equation}
The location of the minimum is summarized in Table~\ref{tab:minvals}. 
\end{thm}


To prove Theorem~\ref{thm:equiv}, we could directly relate the drift
profile to the value of $\min_{\theta\in[0,\pi/2]} K_\mS(\theta)$, and
show that this matches the values obtained by Fayolle and Raschel.

The different cases are summarized in Table~\ref{tab:minvals}.
It is possible to prove the results in Table~\ref{tab:minvals}
by considering the two equations/inequalities that arise from drift
profiles and follow the implications in a straightforward way to
deduce the sign of both $\ln \alpha$ and $\ln \beta$. Some general
case reductions can simplify some of the work. As the inequality
manipulations are rather tedious, we do not include them here. It is
also possible to simply test the 79 small step cases in order to
verify the result.

\section{Extensions and applications}
\label{sec:extensions}
Our half plane bounding strategy does not rely on the size of the
steps nor the convex cone in which the paths are restricted.  Naive
numerical calculations on examples of quarter plane walks with larger
steps and of three dimensional models, so far tested in the
non-negative octant, suggest the bounds remain tight.

Furthermore, in a recent study of three dimensional
walks~\cite{BoBoKaMe15}, Bostan, Bousquet-M\'elou, Kauers, and Melczer
guessed differential equations satisfied by the generating functions
of some small step models, and using this they were able to conjecture
the exact growth constants.  We verified that, in each of the cases
for which they had data, the growth constant for the model $O(\mS)$,
where $O = \mathbb{R}_{\geq 0}^3$, is equal to the minimal bound.

These observations led us to the following conjecture.

\begin{conjecture}
  Let $\mS\subset \mathbb{Z}^d$ be a finite multiset of steps.  Let
  $K_\mS$ be the growth constant for the enumerative sequence counting
  the number of walks restricted to the first orthant. Let $\mathcal{P}$
  be the set of hyperplanes through the origin in $\mathbb{R}^d$ which
  do not meet the interior of the first orthant.  Given $p\in
  \mathcal{P}$ let $K_{\mS}(p)$ be the growth constant of the walks on
  $\mS$ which are restricted to the side of $p$ which includes the first
  orthant.  Then
  \[
    K_\mS = \min_{p\in\mathcal{P}} K_{\mS}(p).
  \]
\end{conjecture}
Note that in all dimensions, $K_{\mS}(p)$ can be computed by
projecting the steps onto the normal to $p$, and enumerating the
resulting unidimensional model.  Thus, if true, our conjecture would
give an elementary way to understand and compute the growth constant
of any class of first orthant restricted walks.

The work of Garbit and Raschel translates this into a probabilistic
context, and in particular they have proved for it in Corollary 9
of~\cite{GaRa16} under the conditions on models which essentially
avoid singular, and trivial models.

Our half plane interpretation has important implications for random
generation.  A naive rejection strategy to generate quarter plane
walks might first generate walks in the whole plane, and reject them
as they leave the quarter plane. Models with a $(-,-)$ drift profile
perform rather poorly in this scheme. However, when the exponential
growth rate of a class of quarter plane walks is the same as some
class of half plane walks which contain it, the following rejection
scheme is provably efficient. The half plane walks with rational slope
form an algebraic language, and are easily generated. Rejecting walks
from this class when they leave the quarter-plane is remarkably
efficient. There are additional details required when the slope is not
rational, and the theory is developed by Lumbroso, Mishna and
Ponty~\cite{LuMiPo16}.

Approximations such as we compute here can aid other direct
strategies. For example, the strategy of diagonals used by Melczer,
Mishna and Wilson relies on determining a set of critical points to
set up the intergral computations. Having a tight bound on the
exponential growth in hand is useful in this process, as soem
candidates can be eliminated immediately.

\section{Acknowledgments}
We are grateful to Andrew Rechnitzer for some initial discussions on
using the upper half plane as an upper bound, in addition to useful
comments from Kilian Raschel, Michael Renardy, Philippe Biane, and
Jean-Fran\c cois Marckert and J\'er\'emie Lumbroso. MM is also very
grateful to LaBRI, Universit\'e de Bordeaux, which acted as a host
institution during the completion of this work.  

\bibliographystyle{plain}
\bibliography{JoMiYe}

\begin{thebibliography}{10}

\bibitem{BaFl02}
C.~Banderier and Ph. Flajolet.
\newblock Basic analytic combinatorics of directed lattice paths.
\newblock {\em Theoret. Comput. Sci.}, 281(1-2):37--80, 2002.
\newblock Selected papers in honour of Maurice Nivat.

\bibitem{BoChHoKaPe16}
A.~Bostan, F.~Chyzak, M.~van Hoeij, M.~Kauers, and L.~Pech.
\newblock Hypergeometric expressions for generating functions of walks with
  small steps in the quarter plane.
\newblock pages 1--30, 2016.
\newblock http://arxiv.org/abs/1606.02982.

\bibitem{BoKa09}
A.~Bostan and M.~Kauers.
\newblock Automatic classification of restricted lattice walks.
\newblock In {\em DMTCS Proceedings of the 21st International Conference on
  Formal Power Series and Algebraic Combinatorics (FPSAC'09), Hagenberg,
  Austria}, pages 203--217, 2009.

\bibitem{BoRaSa14}
A.~Bostan, K.~Raschel, and B.~Salvy.
\newblock Non-{D}-finite excursions in the quarter plane.
\newblock {\em J. Combin. Theory Ser. A}, 121:45--63, 2014.

\bibitem{BoMi10}
M.~Bousquet-M{\'e}lou and M.~Mishna.
\newblock Walks with small steps in the quarter plane.
\newblock In {\em Algorithmic probability and combinatorics}, volume 520 of
  {\em Contemp. Math.}, pages 1--39. Amer. Math. Soc., Providence, RI, 2010.

\bibitem{DeWa15}
D.~Denisov and V.~Wachtel.
\newblock Random walks in cones.
\newblock {\em Ann. Probab.}, 43(3):992--1044, 2015.

\bibitem{Done89}
R.~A. Doney.
\newblock On the asymptotic behaviour of first passage times for transient
  random walk.
\newblock {\em Probab. Theory Related Fields}, 81(2):239--246, 1989.

\bibitem{Dura14}
J.~Duraj.
\newblock Random walks in cones: the case of nonzero drift.
\newblock {\em Stochastic Process. Appl.}, 124(4):1503--1518, 2014.

\bibitem{FaRa12}
G.~Fayolle and K.~Raschel.
\newblock Some exact asymptotics in the counting of walks in the quarter plane.
\newblock In {\em 23rd {I}ntern. {M}eeting on {P}robabilistic, {C}ombinatorial,
  and {A}symptotic {M}ethods for the {A}nalysis of {A}lgorithms ({A}of{A}'12)},
  Discrete Math. Theor. Comput. Sci. Proc., AQ, pages 109--124. Assoc. Discrete
  Math. Theor. Comput. Sci., Nancy, 2012.

\bibitem{FlSe09}
Ph. Flajolet and R.~Sedgewick.
\newblock {\em Analytic combinatorics}.
\newblock Cambridge University Press, Cambridge, 2009.

\bibitem{GaMuRa16}
R.~Garbit, S.~Mustapha, and K.~Raschel.
\newblock Random walks with drift in cones.
\newblock {\em In preparation}, 2016.

\bibitem{GaRa16}
R.~Garbit and K.~Raschel.
\newblock On the exit time from a cone for random walks with drift.
\newblock {\em Rev. Mat. Iberoamericana}, 32(2):511--532, 2016.

\bibitem{Hill48}
E.~Hille.
\newblock {\em Functional {A}nalysis and {S}emi-{G}roups}.
\newblock American Mathematical Society Colloquium Publications, vol. 31.
  American Mathematical Society, New York, 1948.

\bibitem{LuMiPo16}
J.~Lumbroso, M.~Mishna, and Y.~Ponty.
\newblock Taming reluctant random walks in the positive quadrant.
\newblock In Jean-Marc F{\'e}dou, editor, {\em Proceedings of International
  conference on random generation of combinatorial structures GASCom'16},
  Electronic Notes in Discrete Mathematics. Elsevier, 2016.
\newblock To appear.

\bibitem{MeMi14a}
S.~Melczer and M.~Mishna.
\newblock Singularity analysis via the iterated kernel method.
\newblock {\em Combin. Probab. Comput.}, 23(5):861--888, 2014.

\bibitem{MeMi16}
S.~Melczer and M.~Mishna.
\newblock Asymptotic lattice path enumeration using diagonals.
\newblock {\em Algorithmica}, 75:782--811, 2016.

\bibitem{MeWi16}
S.~Melczer and M.~Wilson.
\newblock Asymptotics of lattice walks via analytic combinatorics in several
  variables.
\newblock In {\em Proceedings of FPSAC'16}, Discrete Math. Theor. Comput. Sci.
  Proc., 2016.

\bibitem{MiRe09}
M.~Mishna and A.~Rechnitzer.
\newblock Two non-holonomic lattice walks in the quarter plane.
\newblock {\em Theoret. Comput. Sci.}, 410(38-40):3616--3630, 2009.

\end{thebibliography}

\end{document}